\documentclass[12pt]{amsart}
\usepackage{graphicx,verbatim,amssymb,amscd}
\usepackage[active]{srcltx}

\newtheorem{thm}{Theorem}[section]
\newtheorem{cor}[thm]{Corollary}
\newtheorem{lem}[thm]{Lemma}
\newtheorem{prop}[thm]{Proposition}
\theoremstyle{definition}

\theoremstyle{remark}

\newtheorem{ex}[thm]{Example}
\numberwithin{equation}{section}

\newcommand{\R}{\mathbb R}

\def\CP{\mathbb{C}P}
\def\C{\mathbb{C}}
\def\Z{\mathbb{Z}}
\newcommand{\eps}{\varepsilon}

\def\d{\partial}

\def\0{\varnothing}
\def\disk{\mathbb{D}}
\def\ol{\overline}
\def\uc{\mathbb{S}^1}
\def\si{\sigma}
  \def\Bc{\mathcal{B}}  \def\Cc{\mathcal{C}}
  \def\Lc{\mathcal{L}}  \def\Hc{\mathcal{H}}
\def\Oc{\mathcal{O}}  \def\Rc{\mathcal{R}}  \def\Uc{\mathcal{U}}
\def\Zc{\mathcal{Z}}
\def\Aut{{\rm Aut}}

\begin{document}

\title[Captures, matings and regluings]{Captures, matings and regluings}
\author{Inna Mashanova}
\author{Vladlen Timorin}
\address[I. Mashanova and V. Timorin]
{Faculty of Mathematics and Laboratory of Algebraic Geometry\\
National Research University Higher School of Economics\\
7 Vavilova St 117312 Moscow, Russia}

\address[V.~Timorin]
{Independent University of Moscow\\
Bolshoy Vlasyevskiy Pereulok 11, 119002 Moscow, Russia}
\email{vtimorin@hse.ru}

\thanks{Both authors were partially supported by AG Laboratory
NRU-HSE, MESRF grant ag. 11 11.G34.31.0023}

\thanks{
The second named author was partially supported by
the Deligne fellowship, the Simons-IUM fellowship, RFBR grants 10-01-00739-a,
11-01-00654-a, and MESRF grant MK-2790.2011.1.
}


\begin{abstract}
In parameter slices of quadratic rational functions, we identify arcs
represented by matings of quadratic polynomials.
These arcs are on the boundaries of hyperbolic components.
\end{abstract}

\maketitle
\section{Introduction}

The operation of {\em mating} has been introduced by Douady and Hubbard.
Mating can be applied to a pair of polynomials of the same degree, and
gives a continuous self-map (the {\em mating map}) of some topological space
(the {\em mating space}).
In many cases, the mating space is homeomorphic to the 2-sphere, and the
mating map is a branched covering topologically conjugate to a rational
function.
In this paper, we only consider matings of degree 2 polynomials.

Recall that any quadratic polynomial in one complex variable is affinely conjugate
to $p_c(z)=z^2+c$ for some $c$.
The {\em filled Julia set $K_c$} of $p_c$ is defined as the set of all points $z\in\C$
such that the $p_c$-orbit of $z$ is bounded, and the {\em Julia set} $J_c$
as the boundary of $K_c$.
The {\em Fatou set} of $p_c$ is defined as the complement to $J_c$
in the Riemann sphere; its connected components are called {\em Fatou components}.
The {\em Mandelbrot set $M$} is the locus of all $c$ such that $K_c$ is connected,
equivalently, $c\in K_c$.
Consider two quadratic polynomials $p_c$ and $p_{c'}$ such that $K_c$ and $K_{c'}$
are locally connected, equivalently, $J_c$ and $J_{c'}$ are locally connected.
In this case, it is well known that there are {\em Caratheodory loops}
$\gamma_c:\R/\Z\to J_c$
and $\gamma_{c'}:\R/\Z\to J_{c'}$ that semi-conjugate the map $\theta\mapsto 2\theta$
on $\R/\Z$ with the maps $p_c$ and $p_{c'}$ on $J_c$ and $J_{c'}$, respectively.
Define the mating space $X=X_{c,c'}$ as the quotient of $K_c\sqcup K_{c'}$ by the
minimal equivalence relation $\sim=\sim_{c,c'}$ such that $z\in K_c$ is
equivalent to $z'\in K_{c'}$
if $z=\gamma_c(\theta)$ and $z'=\gamma_{c'}(-\theta)$ for some $\theta\in\R/\Z$.
Since the self-map of $K_c\sqcup K_{c'}$ acting as $p_c$ on $K_c$ and as $p_{c'}$
on $K_{c'}$ takes $\sim$-classes to $\sim$-classes, it descends to a self-map
$p_c\sqcup p_{c'}$ of the mating space $X$.
This map is called the {\em topological mating} of $p_c$ and $p_{c'}$.
If a rational function $f:\CP^1\to\CP^1$ is topologically conjugate to
the topological mating $p_c\sqcup p_{c'}$, then this rational function is called
a {\em conformal mating} of $p_c$ and $p_{c'}$.

In this paper, we identify certain arcs in parameter slices of
quadratic rational functions that consist of conformal matings.
More precisely, we consider the slices $Per_k(0)$ consisting
of conformal conjugacy classes of degree 2 rational functions $f$ with
marked critical points $c_1$, $c_2$ such that $f^{\circ k}(c_1)=c_1$.
These slices were first defined and studied by M. Rees \cite{Rees_description}
and J. Milnor \cite{Milnor-QuadRat}.
In his thesis, B. Wittner \cite{Wittner} described an operation that
provides topological models for many hyperbolic components of $Per_k(0)$.
These are called {\em capture components}.
In this paper, we prove that the boundaries of capture components
of $Per_k(0)$ contain arcs of matings.
The main theorems are Theorem \ref{t:same-reg}, \ref{t:end-cap} and \ref{t:cut-cap}.

\paragraph{\em Organization of the paper}
A significant part of this paper is expository.
In Section 2, we recall the terminology of quadratic invariant laminations
\cite{Thurston}.
We also use this terminology to give several equivalent definitions of
matings and to describe topological models for captures \cite{Rees_description}.
In Section 3, we recall the topological surgery called {\em regluing} \cite{T09}.
Regluing will be used to redescribe topological models for captures.
Finally, in Section 4, we consider parameter slices $Per_k(0)$ of quadratic
rational functions.
Topological models for rational functions representing boundary
points of some hyperbolic components in $Per_k(0)$ were described in \cite{T09}
in terms of regluing.
Comparing this description with the description of captures as regluings
of matings, we obtain arcs of matings on the boundaries of capture components.

\section{Laminations, matings and captures}
In this section, we discuss topological models for quadratic polynomials
and matings of quadratic polynomials.
We also define captures.

\subsection{Invariant laminations}
\label{ss:inv-lam}
Topological models for quadratic polynomials can be described in
terms of Thurston's invariant quadratic laminations in the disk
$\disk=\{z\in\C\,|\, |z|<1\}$.
We first consider a slightly more general notion.

Let $\Omega$ be an open simply connected domain in $\CP^1$, whose
complement consists of more than one point (hence of infinitely many points).
Then $\Omega$ is conformally isomorphic to $\disk$.
We will also assume that the boundary of $\Omega$ is locally connected.
Then the closure in $\CP^1$ of any geodesic in $\Omega$ with respect to the Poincar\'e
metric on $\Omega$ consists of the geodesic itself and at most two limit points of it
that belong to $\d\Omega$.
The closure of a geodesic in $\Omega$ is called a {\em geodesic chord} of $\Omega$.
A {\em geodesic lamination} in $\Omega\subset\CP^1$ is a set of
geodesic chords in $\Omega$, whose union is closed.
Elements of a geodesic lamination $\Lc$ are called {\em leaves of} $\Lc$.
Let $Z$ and $W$ be prime ends of $\Omega$.
We will write $(ZW)_\Omega$ for the geodesic chord of
$\Omega$ connecting the prime ends $Z$ and $W$.
We allow for $Z=W$, in which case $(ZW)_\Omega$ is defined as the
single point, which is the prime end impression of $Z=W$,
and called a {\em degenerate leaf} of $\Lc$.
For convenience, we will assume that every geodesic lamination
contains all degenerate leaves $(ZZ)_\Omega$.
Note that there may be different prime ends $Z\ne W$ with the
same prime end impression; they still define a non-degenerate leaf $(ZW)_\Omega$,
which is a closed curve.

Let $f:\d\Omega\to\d\Omega$ be any continuous map that extends to
the set $\Omega$ as a proper orientation-preserving branched covering of degree two.
Then $f$ acts on prime ends of $\Omega$.
We say that a lamination $\Lc$ in $\Omega$ is {\em forward invariant} with respect
to $f$ if, for every leaf $\ell=(ZW)_\Omega\in\Lc$, the curve
$f[\ell]=(f(Z)f(W))_\Omega$ is also a leaf of $\Lc$, possibly degenerate.
We use the square brackets in the notation $f[\ell]$ to emphasize that
this curve is, in general, different from the image of the curve $\ell$
under the map $f$.

We now define the notion of an $f$-invariant lamination in $\Omega$.
This is a forward $f$-invariant lamination $\Lc$ that satisfies the
following properties:
for every leaf $\ell\in\Lc$, there is another leaf $\tilde\ell$ of $\Lc$ such that
$f[\ell]=f[\tilde\ell]$, and there are two leaves $\ell_1$ and $\ell_2$
such that $f[\ell_1]=f[\ell_2]=\ell$.

People usually consider laminations in the unit disk $\disk$
that are invariant under the map $\si_2:z\mapsto z^2$.
These laminations are called (Thurston) {\em quadratic invariant laminations}.
The boundary of the unit disk is the unit circle $\uc=\{z\in\C\,|\, |z|=1\}$.
The unit circle $\uc$ is identified with $\R/\Z$ by means of the map
$$
\theta\in\R/\Z\ \mapsto\ \ol{\theta}=e^{2\pi i\theta}\in\uc.
$$
If $z=\ol{a}$ and $w=\ol{b}$, then we write $\ol{ab}$ or $zw$ for
the leaf connecting $z$ with $w$.
In the case $\Omega=\disk$, we will identify prime ends with points
on the unit circle that are the corresponding prime end impressions.

Let $\Lc$ be a geodesic lamination in the unit disk, and $\Omega$
an arbitrary simply connected domain in $\CP^1$, whose complement
contains more than one point and whose boundary is locally connected.
We can transform the lamination $\Lc$ into a geodesic lamination
in $\Omega$ as follows.
Let $\phi:\disk\to\Omega$ be a Riemann map.
Suppose that we fixed the Riemann map, i.e. we specified the
point $\phi(0)$ and the argument of the derivative $\phi'(0)$.
By Caratheodory's theory, the map $\phi$ acts on prime ends, i.e.
$\phi(z)$ is a well-defined prime end in $\Omega$ for every $z\in\uc$.
With every leaf $\ell=zw\in\Lc$, we associate the curve
$\ell_\Omega=(\phi(z)\phi(w))_\Omega$.
The set of all such curves is a geodesic lamination $\Lc(\Omega,\phi)$ in $\Omega$.
We call this lamination {\em the ($\phi$-)image of the lamination $\Lc$ in $\Omega$}.
Sometimes, we write simply $\Lc(\Omega)$ if the choice of the Riemann map
is clear.
Clearly, any geodesic lamination in $\Omega$ is the image of some
geodesic lamination in the unit disk.
Moreover, if $f:\ol\Omega\to\ol\Omega$ is a continuous map such that
$f$ is holomorphic on $\Omega$ and has degree two, and the continuous
extension $\ol\phi:\ol\disk\to\ol\Omega$ of a Riemann map
$\phi:\disk\to\Omega$ semi-conjugates the map $\si_2$ on $\uc$ with the map
$f|_{\d\Omega}$, then any $f$-invariant geodesic lamination in $\Omega$ is
the $\phi$-image of some invariant quadratic lamination.

We now introduce some notions for laminations in the disk.
By the construction just described, they automatically carry over to
laminations in any simply connected domain $\Omega\subset\CP^1$,
whose complement consists of more than one point, and whose boundary
is locally connected.
Let $\Lc$ be a geodesic lamination in $\disk$.
{\em Gaps} of $\Lc$ are defined as closures (in $\CP^1$) of
connected components of $\disk-\bigcup\Lc$.
The {\em basis} of a gap $G$ is defined as $G'=\uc\cap G$.
Clearly, a gap is uniquely determined by its basis.
Gaps may be {\em finite} or {\em infinite} according to whether their
bases are finite or infinite.
The lamination $\Lc$ is called {\em clean} if, whenever two leaves of $\Lc$
share an endpoint, they are on the boundary of a finite gap.
If $\Lc$ is clean, then it defines an equivalence relation $\sim$ on $\CP^1$:
two different points $z$ and $w$ are equivalent if they belong
to the same leaf or a finite gap of $\Lc$.
One can prove using Moore's theorem \cite{Moore} that the quotient space $\CP^1/\sim$
is homeomorphic to the sphere.
Let $J_{\Lc}$ be the image of $\uc$ under the quotient projection
$\CP^1\to\CP^1/\sim$.
If $\Lc$ is a quadratic invariant lamination, then the map
$\si_2$ descends to a self-map $F_\Lc$ of $J_\Lc$.
Actually, the map $F_\Lc$ can be extended to the entire sphere $\CP^1/\sim$
as a branched covering (we keep the same notation $F_\Lc$ for the extended map).
This branched covering is called a {\em topological polynomial}.

Let $p_c(z)=z^2+c$ be a polynomial, whose Julia set $J_c$ is locally connected.
Then $p_c$ is topologically conjugate to a topological polynomial $F_\Lc$
corresponding to an invariant lamination $\Lc$.
In this sense, we will say that the lamination $\Lc$ {\em models} the polynomial $p_c$.

\begin{ex}
 Consider the quadratic polynomial $z^2-1$.
 Its critical point $0$ is periodic of period 2: $0\mapsto -1\mapsto 0$.
 The lamination $\Lc$ that models $z^2-1$ can be constructed as follows.
 Consider the chord $\ell_0=\ol{\frac 13\frac 23}$.
 Note that $\si_2[\ell_0]=\ell_0$.
 Set $\ell_1=-\ell_0$, i.e. the leaf $\ell_1$ is obtained from $\ell_0$ by
 the half-turn around the origin.
 The chords $\ell_0$ and $\ell_1$ divide the unit circle into four arcs.
 Divide each of these four arcs into four equal parts and connect
 the $1/4$ with the $3/4$ of this subdivision by an arc.
 In this way, we obtain 6 chords that divide the unit circle into 12 arcs.
 Repeat the same operation with each of these 12 arcs, and so on.
 All chords that can be obtained by this countable procedure form
 a $\si_2$-invariant lamination $\Lc$ that models $z^2-1$.
\end{ex}

\subsection{Critical leaf laminations}
Consider a point $\ol\vartheta$ in the unit circle.
It defines a diameter $\ell_0$ of the unit circle, namely, the geodesic chord connecting
the two $\sigma_2$-preimages $\ol{\frac\vartheta 2}$ and
$\ol{\frac{\vartheta+1}2}$ of $\ol\vartheta$.
We can now define an invariant quadratic lamination $\Lc_\vartheta$ as follows.
A {\em pullback} of $\ell_0$ is defined as any geodesic chord $zw$ such that,
for some integer $m>0$, we have $\si_2^{\circ m}(z)\si_2^{\circ m}(w)=\ell_0$,
and for all $i<m$, the geodesic chord $\si_2^{\circ i}(z)\si_2^{\circ i}(w)$
does not cross the leaf $\ell_0$ in $\disk$ (although it may have
an endpoint in common with $\ell_0$).
Consider the set of all pullbacks of $\ell_0$, and also all geodesic chords
obtained as limits of pullbacks with respect to the Hausdorff metric.
We obtain a quadratic invariant lamination $\Lc_\vartheta$, not necessarily clean.
The lamination $\Lc_\vartheta$ is called the {\em critical leaf lamination}.

There are two cases, in which the lamination $\Lc_\vartheta$ is unclean.

{\em Case 1:} one endpoint of $\ell_0$ is $\si_2$-periodic.
In this case, there is an infinite concatenation of leaves $\ell_0$,
$\ell_1$, $\ell_2$, $\dots$ of $\Lc_\vartheta$ such that
$\ell_{i+1}$ shares an endpoint with $\ell_i$ for every $i=0$, $1$, $2$, $\dots$.
There is also a periodic leaf such that one
of its endpoints is an endpoint of $\ell_0$, and the other
endpoint is the limit of the leaves $\ell_i$ as $i\to\infty$.
The infinite concatenation of leaves $\ell_i$ is called a {\em caterpillar}.
To make the unclean lamination $\Lc_\vartheta$ into a clean lamination,
we can apply the {\em cleaning procedure}, i.e. remove the caterpillar and all
its pullbacks.
What remains is a clean lamination $\Lc^c_\vartheta$ called the
{\em cleaning} of $\Lc_\vartheta$.

{\em Case 2:} no endpoint of $\ell_0$ is periodic; however, there are two
finite gaps of $\Lc_\vartheta$ such that $\ell_0$ is their common edge.
In this case, the cleaning of the lamination $\Lc_\vartheta$ is
the removal of the critical leaf $\ell_0$ and all its pullbacks.
Thus all leaves that survive this cleaning procedure are limits
of pullbacks of $\ell_0$; they form a clean lamination $\Lc_\vartheta^c$.
The lamination $\Lc_\vartheta^c$ has a finite critical gap, i.e.
a finite gap containing a pair of opposite (different by a half-turn)
points on the boundary.
It is proved in \cite{Thurston} that this gap is necessarily preperiodic.

Cleanings $\Lc^c_\vartheta$ of laminations $\Lc_\vartheta$,
whose critical leaves have no periodic endpoints, have the following meaning.
Consider a polynomial $p_c$ such that $c$ is in the Julia set $J_c$,
and the Julia set is locally connected.
Then $p_c$ is modeled by the lamination $\Lc^c_\vartheta$, where
the angle $\vartheta$ is chosen so that $\gamma_c(\vartheta)=c$
(there may be several angles satisfying this equality, they give rise to
the same cleaning $\Lc^c_\vartheta$).
The cleanings $\Lc_\vartheta^c$ of critical leaf laminations $\Lc_\vartheta$,
whose critical leaves have a periodic endpoint,
model polynomials, whose parameters are in the interior of
the Mandelbrot set.

\subsection{Matings}
An equivalent definition of a topological mating is the following.
Let $p_{c_1}$ and $p_{c_2}$ be quadratic polynomials with connected
and locally connected Julia sets.
Then $p_{c_1}$ is modeled by some quadratic invariant lamination $\Lc_1$,
and $p_{c_2}$ is modeled by some quadratic invariant lamination $\Lc_2$.
We can draw the leaves of $\Lc_1$ in the unit disk, and the leaves of $\Lc_2$
in the complement to $\disk$ in $\CP^1$ (which is also a disk), i.e.
we consider laminations $\Lc_1$ and $\Lc_2^{-1}=\Lc_2(\CP^1-\ol\disk)$.
The lamination $\Lc_2(\CP^1-\ol\disk)$ is formed using the map
$z\mapsto 1/z$ as the Riemann map for $\CP^1-\ol\disk$.
Thus the leaves of $\Lc_2^{-1}$ are images of the leaves of $\Lc_2$ under
the map $z\mapsto 1/z$.
We write $\ell^{-1}$ instead of $\ell_{\CP^1-\ol\disk}$ for the
image of the leaf $\ell\in\Lc_2$.
Let $\sim_1$ be the equivalence relation associated with $\Lc_1$,
and $\sim_2$ the equivalence relation associated with $\Lc_2^{-1}$
(i.e. two different points $z$ and $w$ are equivalent with respect to $\sim_2$
if they lie in the same leaf or finite gap of $\Lc_2^{-1}$).
Finally, let $\sim$ be the minimal equivalence relation on $\CP^1$ containing both
$\sim_1$ and $\sim_2$.
The mating space $X_{c_1,c_2}$ of $p_{c_1}$ and $p_{c_2}$ identifies with the
space $\CP^1/\sim$.
Clearly, the image of $\uc$ in $X_{c_1,c_2}$ is a quotient of $J_{\Lc_1}$, and,
simultaneously, a quotient of $J_{\Lc_2}$.
The map $\si_2$ descends to a continuous self-map of the quotient space
$\uc/\sim\subset X_{c_1,c_2}$.
This map coincides with the restriction of the mating map to $\uc/\sim$.
The set $\uc/\sim$, which can also be obtained by pasting the
Julia sets $J_{c_1}$ and $J_{c_2}$ together as described in the introduction
is called the {\em Julia set of the mating}.
If the mating is topologically conjugate to a rational function,
then the conjugacy takes the Julia set of the mating to the
Julia set of this rational function.
Note that this construction explains the minus sign in the definition
of matings: the image of a point $\ol\theta$ on the unit circle under the map
$z\mapsto 1/z$ is the point $\ol{-\theta}$.

The two definitions of matings can be combined into the following non-symmetric
construction.
Consider the filled Julia set $K_{c_1}$ of the first polynomial $p_{c_1}$,
and the quadratic invariant lamination $\Lc_2$ that models the second
polynomial $p_{c_2}$.
By the B\"ottcher theorem, there exists a unique Riemann map
$\phi_{c_1}:\disk\to\CP^1-K_{c_1}$ that conjugates the map $z\mapsto z^2$
with the map $p_{c_1}$ restricted to the basin of infinity.
Clearly, $\phi_{c_1}$ must map $0$ to $\infty$.
We can now take the image of the lamination $\Lc_2$ in the basin
of infinity $\Omega_{c_1}=\CP^1-K_{c_1}$.
Taking the quotient of the sphere by the equivalence relation
generated by the lamination $\Lc_2(\Omega_{c_1})$, we also obtain
a topological model for the mating $p_{c_1}\sqcup p_{c_2}$.
We will use this non-symmetric construction below, when discussing
topological models for captures.

The non-symmetric construction of matings can be generalized to the
case, where $\Lc_2$ does not necessarily model a quadratic polynomial.
Thus we can talk about a mating of a polynomial and an invariant lamination.
Similarly, we can talk about a mating of two invariant laminations.

\subsection{Internal and external angles}
Let $p_c(z)=z^2+c$ be a polynomial such that the critical point $0$
of $p_c$ is periodic of minimal period $k$.
Let $A$ be the Fatou component of $p_c$ containing $0$.
The map $p_c^{\circ k}$ takes $A$ to itself.
By the B\"ottcher theorem, there exists a conformal isomorphism
$\psi:\disk\to A$ that conjugates the map $z\mapsto z^2$ with the restriction
of $p_c^{\circ k}$ to $A$.
It is well-known that the Julia set of $p_c$, as well as the boundary of
any Fatou component of $p_c$, are locally connected.
Therefore, the map $\psi$ admits a continuous extension $\ol\psi:\ol\disk\to\ol A$.
A point of $\d A$ of {\em internal angle} $\varkappa$ is defined as
$\ol\psi(\ol\varkappa)$.
Since $\d A$ is homeomorphic to the circle, the map $\ol\psi$ must
be a homeomorphism (as follows e.g. from Caratheodory's theory), and
then internal angles are in one-to-one correspondence with points of $\d A$.

If $B$ is any other bounded Fatou component of $p_c$, then there is a minimal
integer $m$ such that $p_c^{\circ m}(B)=A$.
The map $p_c^{\circ m}$ is then a homeomorphism between the closure of
$B$ and the closure of $A$.
Using this homeomorphism, we define internal angles for the boundary
points of $B$.
Thus any point on the boundary of any bounded Fatou component of $p_c$
has a well-defined internal angle with respect to this bounded Fatou component.

On the other hand, any point $z$ of $J_c$ has the form $\gamma_c(\theta)$
for some angle $\theta\in\R/\Z$.
This angle is called an {\em external angle} of $z$.
The same point can have several different external angles.
If the point $z$ is on the boundary of some bounded Fatou component of $p_c$,
then there are two ways of identifying the point $z$:  1) we can just
specify the external angle of $z$, and 2) we can specify the Fatou component,
whose boundary contains the point $z$ and the internal angle of $z$ with
respect to this component.

\subsection{Captures}
Capture is an operation making polynomials into (models of) rational functions.
It was first introduced in the thesis of B. Wittner \cite{Wittner} in 1988.
Similar to matings, captures can be defined in terms of topological models.
However, an easier definition (due to M. Rees \cite{Rees_description})
uses combinatorial equivalence classes.
Combinatorial equivalence is a certain equivalence relation on the
set of orientation preserving topological branched self-coverings of the
sphere that are {\em critically finite}, i.e. every critical point gets
eventually mapped to a periodic cycle (such coverings are called
{\em Thurston maps}).
The post-critical set $P_f$ of a Thurston map $f$ is defined as the minimal
forward invariant set containing all critical values.
Two Thurston maps $f$ and $g$ are called {\em combinatorially equivalent}
(or {\em Thurston equivalent}) if $f$ is homotopic, relative to the set $P_f$,
to a Thurston map $h\circ g\circ h^{-1}$ topologically conjugate to $g$
(i.e. $h$ is an orientation preserving self-homeomorphism of the sphere, and
the homotopy connecting $f$ to $h\circ g\circ h^{-1}$ consists of Thurston maps with the
same post-critical set $P_f$).

{\em Thurston's rigidity theorem} claims that, with few exceptions that
can be explicitly described, any combinatorial equivalence class of
Thurston maps contains at most one rational function.
In particular, this is true for {\em hyperbolic Thurston maps}, i.e. Thurston maps
such that every critical point gets eventually mapped to a cycle containing
a critical point.
Thus a combinatorial equivalence class of hyperbolic Thurston maps
either contains exactly one rational function, or contains no rational functions
at all.
{\em Thurston's characterization theorem}
\cite{DH} provides a topological criterion
distinguishing these two cases.

Consider a critically finite polynomial $p_c(z)=z^2+c$ such that $0$ is
a periodic point of $p_c$ of some minimal period $k$.
Let $v$ be some strictly preperiodic point of $p_c$ that gets eventually
mapped to $0$.
Then $v$ lies in some interior component $V$ of $K_c$.
Let $O(v)$ denote the forward orbit of $v$.
By our assumption, $0\in O(v)$ but $O(v)$ is different from the orbit of $0$.
Let us choose a point $b$ on the boundary of $V$.
The operation of capture is almost determined by the choice of the two points
$v$ and $b$.
Let $\beta:[0,1]\to\CP^1$ be a simple path with the following properties:
$\beta(0)=\infty$, $\beta(1/2)=b$, $\beta(1)=v$, and the intersection
$\beta[0,1]\cap J_c$ consists of only one point $b$.
A path with these properties is called a {\em capture path}.
What is really enough to know to define a capture is the pair of
points $v$, $b$ plus the homotopy class of the path
$\beta:[0,1/2]\to\Omega_c\cup\{b\}$ with fixed endpoints.
Note that this homotopy class is determined by the choice of an external angle of $b$.
This angle will also be called the {\em external angle of the capture path} $\beta$.
Define a {\em path homeomorphism} $\sigma_\beta$ as a self-homeomorphism
of the sphere that is equal to the identity outside a small neighborhood
of $\beta[0,1]$ (i.e. outside a ``narrow tube'' around $\beta[0,1]$)
and such that $\sigma_\beta(\infty)=v$.

Note that $\sigma_\beta\circ p_c$ is a topological branched covering,
whose homotopy class relative to $O(v)$ is well defined (provided that
the neighborhood of $\beta[0,1]$, in which $\sigma_\beta$ is different
from the identity, is small enough so that it does not intersect
$O(v)$)\footnote{Two branched coverings $f_0$ and $f_1$ are homotopic relative to a
set $O$ if there is a homotopy $\phi_t$, $t\in [0,1]$ consisting of
homeomorphisms such that $\phi_t=id$ on $O$ for all $t$, and
$\phi_0\circ f_0=f_1\circ\phi_1$}.
Note also that this covering is critically finite, with post-critical
set $O(v)$ (the critical point $\infty$ is mapped to $v$ and then gets
eventually mapped to $0$).
Thus the combinatorial class of $\sigma_\beta\circ p_c$ is well defined.
The map $\sigma_\beta\circ p_c$ (or rather its combinatorial class)
is called a {\em formal capture} of $p_c$.
An explicit description of all paths $\beta$, for which the formal capture
is Thurston equivalent to a rational function is known by a result of
Mary Rees and Tan Lei.
We will state this result later.
A rational function that is combinatorially equivalent to the formal
capture is called a {\em conformal capture}.

\subsection{Capture paths define matings}
We now start describing topological models for captures.
These models, due to M. Rees \cite{Rees_description}, reveal a close
connection between captures and matings.
Consider a quadratic polynomial $p_c$ such that $p_c^{\circ k}(0)=0$,
a capture path $\beta$ and points $v=\beta(1)$ and $b=\beta(1/2)$.
Let $\vartheta$ be the external angle of the capture path $\beta$.

We now form the critical leaf lamination $\Lc_\vartheta$.
In our case, the angle $\vartheta$ cannot be periodic:

\begin{lem}
\label{l:str-per}
  If $\vartheta\in\R/\Z$ is a periodic angle (with respect to the angle
  doubling map), then the point $\gamma_c(\vartheta)$ cannot belong
  to the boundary of a strictly preperiodic Fatou component of $p_c$.
\end{lem}

\begin{proof}
The polynomial $p_c$ is modeled by some quadratic invariant lamination $\Lc$
in $\disk$.
If the point $\gamma_c(\vartheta)$ belongs to the boundary of some strictly
preperiodic component of $p_c$, then the point $\ol\vartheta$ belongs to the
boundary of some strictly preperiodic gap $G$ of $\Lc$.
Let $m$ be the period length of $\ol\vartheta$.
A certain iterate $s$ of the map $\si_2^{\circ m}$ takes $G$ to a periodic gap
$s(G)$ (this means that the basis $G'$ of the gap $G$
is mapped under $s$ to the basis of some periodic gap, which we denote by $s(G)$).
On the other hand, we have $s(\ol\vartheta)=\ol\vartheta$.
Since $\ol\vartheta$ belongs to the boundaries of two different gaps
$G$ and $s(G)$, there
must be a leaf of $\Lc$ having $\ol\vartheta$ as an endpoint.
Let $g$ be this leaf (if it is not a part of a finite gap) or a finite gap
containing this leaf.
{\em Edges} of the gap (or leaf) $g$ are defined as leaves on the boundary of $g$.
{\em Vertices} of $g$ are defined as points of $g\cap\uc$.
Then $\ol\vartheta$ is one of the vertices of $g$.
Consider {\em co-oriented edges} of $g$, i.e. edges of $g$ equipped with
a choice of an outer side of $g$, a side on which there are no
vertices of $g$ apart from the endpoints of the given edge.
If $g$ is not a leaf, then co-oriented edges are the same as edges:
for every edge, there is only one outer side.
If $g$ is a leaf, then $g$ has two different co-oriented edges, one for
each side of $g$.

Every gap of $\Lc$ adjacent to $g$ defines a co-oriented edge of $g$.
Let $\ell_G$ be the co-oriented edge of $g$ defined by the gap $G$.
Then $\ell_G$ gets mapped to $\ell_{s(G)}$ but $\ell_{s(G)}$ never
maps back to $\ell_G$ under the iterates of $s$.
It follows that $g$ gets eventually mapped under $\si_2$ to a finite critical
gap of $\Lc$.
However, a finite critical gap is always strictly preperiodic, as follows
from \cite{Thurston}.
A contradiction.
\end{proof}

Since $\vartheta$ is not periodic, the endpoints of the critical geodesic chord
$\ell_0=\ol{\frac\vartheta 2\frac{\vartheta+1}2}$ are not periodic either.
Therefore, $\Lc_\vartheta$ is either a clean lamination with the critical leaf
$\ell_0$, or an unclean lamination, whose cleaning has a finite gap containing $\ell_0$.
Consider the lamination $\Lc_\vartheta(\Omega_c)$.
The lamination $\Lc_\vartheta(\Omega_c)$ defines an equivalence relation
$\approx_{c,\beta}$ on $\CP^1$.
This is the minimal equivalence relation such that every leaf and every
finite gap of $\Lc_\vartheta(\Omega_c)$ belongs to some equivalence class.
The quotient space of $\CP^1$ by the equivalence relation $\approx_{c,\beta}$
(together with a natural map defined on this space) is the mating
$p_{c_1}\sqcup\Lc_\vartheta$ of the polynomial $p_{c_1}$ and the lamination
$\Lc_\vartheta$.

Recall from \cite{Thurston} that $\Lc_\vartheta$ has a
unique finite invariant gap or non-degenerate leaf.
We will call this gap or leaf the {\em central gap} of $\Lc_\vartheta$.
We can now state the result of M. Rees and Tan Lei \cite{TanLei}:

\begin{thm}
  The Thurston map $\sigma_\beta\circ p_c$ is combinatorially equivalent
  to a rational function if and only if the image of the central gap of
  $\Lc_\vartheta$ in $\Lc_\vartheta(\Omega_c)$ does not separate the sphere.
\end{thm}

We will need the following lemma:

\begin{lem}
\label{l:no-per}
  Suppose that the image of the central gap of $\Lc_\vartheta$ in
  $\Lc_\vartheta(\Omega_c)$ does not separate the sphere.
  Then no periodic leaf of $\Lc_\vartheta(\Omega_c)$ is a closed curve.
\end{lem}

The statement of the lemma follows from topological models for captures
described in \cite{Rees_description}, see Subsection \ref{ss:topcap}.
However, we give a more direct proof here.

\begin{proof}
  We first recall a general statement about orbit portraits.
  An {\em orbit portrait} can be defined as a $\si_2$-periodic cycle
  of geodesic chords of $\disk$ that have no intersection points in $\disk$.
  Every orbit portrait $\Oc$ defines a {\em wake} in the parameter plane
  of complex polynomials.
  The wake corresponding to an orbit portrait $\Oc$ consists of all
  parameter values $c$ such that, for every $\ol{ab}\in\Oc$, the
  external rays of angles $a$ and $b$ in the dynamical plane of the polynomial $p_c$
  land at the same point.
  It is proved in \cite{Milnor-Mand} that every wake is bounded by
  two external parameter rays that land at the same point (this point is
  called the {\em root point} of the wake).
  We say that two orbit portraits {\em co-exist} if they either coincide
  or have no intersection points in $\disk$.
  Note that every finite gap or leaf, whose vertices are permuted by
  $\si_2$ preserving their cyclic order, defines an orbit portrait.
  Orbit portraits obtained in this way are called {\em principal orbit portraits}.
  Every orbit portrait co-exists with exactly one principal orbit portrait.
  This classical statement can be easily proved either by methods
  of \cite{Milnor-Mand} or using the minor leaf theory of \cite{Thurston}.
  In terms of the parameter plane of complex polynomials, this statement
  means that every wake lies in a {\em principal wake}, whose root point
  belongs to the main cardioid of the Mandelbrot set.

  Suppose now that some periodic leaf $\ell_{\Omega_c}$ of $\Lc_\vartheta(\Omega_c)$
  is a closed curve.
  The corresponding periodic leaf $\ell$ of $\Lc_\vartheta$ defines
  an orbit portrait $\Oc$.
  Since $\ell_{\Omega_c}$ is a closed curve, the geodesic chord $\ell^*$
  that is obtained from $\ell$ by complex conjugation must belong to
  the lamination $\Lc$ that models the polynomial $p_c$.
  Therefore, the lamination $\Lc$ contains the conjugate orbit portrait $\Oc^*$.
  Let $G$ be the central gap of $\Lc_\vartheta$.
  It defines a principal orbit portrait, for which we will use the same letter $G$.
  Obviously, $G$ co-exists with $\Oc$.
  It follows that the complex conjugate principal orbit portrait $G^*$
  co-exists with $\Oc^*$.
  Since there is only one principal orbit portrait co-existing with $\Oc^*$,
  the central gap of $\Lc$ must coincide with $G^*$.
  It follows that the image of the central gap $G$ of $\Lc_\vartheta$ in $\Omega_c$
  disconnects the sphere.
\end{proof}

We will now assume that the Thurston map $\sigma_\beta\circ p_c$ is
combinatorially equivalent to a rational function.
In the next section, we describe a topological model for this
conformal capture due to M. Rees.
We will need the following property of the mating $p_c\sqcup\Lc_\vartheta$:

\begin{prop}
\label{p:A}
 Let $\pi$ be the canonical projection from the filled Julia set $K_c$
 to the mating space of $p_c\sqcup\Lc_\vartheta$.
 If $A$ is the Fatou component of $p_c$ containing the critical point $0$,
 then the restriction of $\pi$ to the closed curve $\d A$ is a homeomorphism.
\end{prop}

It follows that the Fatou component $\pi(A)$ of the mating $p_c\sqcup\Lc_\vartheta$
has a Jordan curve boundary.

\begin{proof}
 Suppose that the restriction of $\pi$ to $\d A$ is not one-to-one.

 {\em Step 1: the critical gap.}
 Let $G$ be the critical gap of the lamination $\Lc$ that models $p_c$,
 i.e. the gap of $\Lc$ containing the origin.
 Then $G$ is periodic of period $k$ and symmetric with respect to the origin.
 Note that $\gamma_c(\d G)=\d A$.
 There is an edge $M$ of $G$ such that $\si_2^{\circ k}[M]=M$.
 Moreover, the {\em hole} of $G$ behind $M$, i.e. the component
 of $\uc-G$ bounded by the endpoints of $M$ and disjoint from $G$,
 is the longest among all holes of $G$
 (the opposite leaf $-M$ is also on the boundary of $G$; its hole has
 the same length; all other holes have strictly smaller length).
 These facts are among the basic properties of {\em minor leaf laminations};
 they are discussed in \cite{Thurston}.
 The leaf $M$ is called the {\em major leaf} of $G$ (and of $\Lc$), or
 simply the {\em major}.

 {\em Step 2: a new invariant lamination.}
 Set $G^*$ to be the set obtained from $G$ by complex conjugation
 (if we want to place both $G$ and $\Lc_\vartheta$ to the same disk,
 then we must take complex conjugation of something --- either of $G$
 or of $\Lc_\vartheta$).
 Let $M^*$ denote the geodesic chord obtained from the major $M$ of $G$ by
 complex conjugation.
 Our assumption that the restriction of $\pi$ to $\d A$ is not
 one-to-one translates as follows: there is a leaf or a finite gap of $\Lc_\vartheta$
 such that two different points of $\d G^*$ are among its vertices.
 There is a natural monotone map $\xi:\uc\to\uc$ that collapses the
 closures of all holes of $G^*$ and that semi-conjugates the map
 $\si_2^{\circ k}:\d G^*\to\d G^*$ (i.e. the map $\si_2^{\circ k}$ restricted
 to the basis of $G^*$ and extended over all edges of $G^*$ in a monotone
 continuous way) with the map $\si_2:\uc\to\uc$.
 For every leaf $\ell=\ol{ab}$, we set $\xi(\ell)=\ol{\xi(a)\xi(b)}$,
 thus the $\xi$-images of leaves are well defined.
 Consider the set of $\xi$-images of all leaves of $\Lc_\vartheta$.
 Denote this set by $\Lc_G$.
 It is not hard to verify that the collection of leaves thus obtained
 is a quadratic invariant lamination.
 It follows from our assumption that the lamination $\Lc_G$ is non-trivial,
 i.e. it contains non-degenerate leaves.

 {\em Step 3: the critical leaf.}
 No leaf of $\Lc_G$ can intersect the $\xi$-image of the critical geodesic chord
 $\ell_0=\ol{\frac\vartheta 2\frac{\vartheta+1}2}$ in $\disk$.
 Note that the geodesic chord $\xi(\ell_0)$ is also a critical chord (i.e. a
 diameter of the unit circle), whose endpoints get eventually mapped to
 the fixed point $\ol 0$, which is the $\xi$-image of $M^*$.
 It follows that $\Lc_G$ is the critical leaf lamination
 generated by the critical leaf $\xi(\ell_0)$.
 Indeed, the chord $\xi(\ell_0)$ must eventually map to a geodesic
 chord containing $\ol 0$.
 However, if an invariant lamination contains any leaf having $\ol 0$
 as an endpoint, then it must contain the leaf $\ol{0\frac 12}$.\footnote{
 Among all leaves $\ol{0a}$ choose the one, for which $\ol a$ is the
 closest to $\ol 0$. Then the leaf $\ol{0(2a)}$ will intersect the leaf
 $\ol{\frac 12(a+\frac 12)}$ unless $a=1/2$.}
 We now need to consider two cases: either the leaf $\xi(\ell_0)$
 coincides with $\ol{0\frac 12}$ or no endpoint of this leaf is periodic
 under the angle doubling map.

 {\em Case 1}: the leaf $\xi(\ell_0)$ coincides with $\ol{0\frac 12}$.
 In this case, the critical leaf lamination generated by $\ol{0\frac 12}$
 is not clean, and all leaves of this lamination are pullbacks of the
 critical leaf.
 It follows that every leaf of $\Lc_G$ connects two points that both
 get eventually mapped to $\ol 0$, but not simultaneously.
 In particular, there is a leaf $\ell$ of $\Lc_\vartheta$, whose $\xi$-image
 coincides with $\ol{0\frac 12}$.
 An endpoint of $\ell$ that projects to $\ol 0$ under $\xi$ must
 lie in the closure of the complementary arc to $G^*$ bounded by $M^*$.
 At the same time, this endpoint belongs to $G^*$ by our assumption.
 Therefore, it coincides with an endpoint of $M^*$.
 Hence the leaf $\ell$ shares endpoints with $M^*$ and with
 $-M^*$ (the centrally symmetric to $M^*$ leaf with respect to the origin).
 Thus there is a quadrilateral such that two edges of it are $M^*$ and $-M^*$,
 and the other two edges $\ell$ and $-\ell$ are leaves of $\Lc_\vartheta$.
 Both $\ell$ and $-\ell$ map to $\si_2(\ell)$, and both $M^*$ and $-M^*$
 map to $\si_2(M^*)$.
 Hence $\si_2(\ell)=\si_2(M^*)$.
 The image of $\ell$ in $\Lc_\vartheta(\Omega_c)$ is a closed curve, since
 both endpoints of $M$ map to the same point under $\gamma_c$.
 Moreover, this closed curve is a periodic leaf of $\Lc_\vartheta(\Omega_c)$,
 since the endpoints of $M$ are periodic.
 This contradicts Lemma \ref{l:no-per}.

 {\em Case 2:} no endpoint of $\xi(\ell_0)$ is periodic
 under the angle doubling map.
 In this case, the lamination $\Lc_G$ is clean or becomes clean after
 removal of the critical leaf and all its pullbacks.
 It has no infinite gaps, i.e. the entire disk $\ol\disk$ is the union
 of leaves and finite gaps of $\Lc_G$.
 It follows that the quotient of $\uc$ by the equivalence relation generated
 by $\Lc_G$ is a {\em dendrite} (i.e. is locally connected and homeomorphic
 to the complement to an open dense topological disk in $\CP^1$).
 It follows that $\pi(\d A)$ must also be a dendrite.
 This is a contradiction, because the complement to $\pi(A)$ contains
 some other Fatou components of the mating $p_c\sqcup\Lc_\vartheta$.
\end{proof}

\section{Regluing and topological captures}
In this section, we recall the basic properties of regluing,
a topological surgery on rational functions introduced in \cite{T09}.
We will also relate regluing to topological models for captures.

\subsection{Topological models for captures}
\label{ss:topcap}
We first describe topological models for captures given in \cite{Rees_description}.
Let $p_c$ be a quadratic polynomial such that $p_c^{\circ k}(0)=0$, and $\beta$ a
capture path for $p_c$ of external angle $\vartheta$.
A topological model for the conformal capture of $p_c$ corresponding to the
capture path $\beta$ is perhaps easier to describe in terms of regluing of the
mating $p_c\sqcup\Lc_\vartheta$.
However, we need to know in advance that the corresponding mating space is
homeomorphic to the sphere.
We will prove this statement later using the topological models from
\cite{Rees_description}.

Let $\Lc$ be the quadratic invariant lamination that models $p_c$,
and $G$ the infinite gap of $\Lc$ that corresponds to the Fatou component
of $p_c$ containing the point $\beta(1)$.
The map $\si_2:\uc\to\uc$ extends to a Thurston map $s_c$ such that
$s_c(\ell)=\si_2[\ell]$ for every leaf $\ell\in\Lc$.
Moreover, we can set $s_c(z)=z^2$ outside the unit disk,
and arrange that $0$ be a critical point of $s_c$ such that $s_c^{\circ k}(0)=0$.
The map $s_c$ is combinatorially equivalent to $p_c$.
Indeed, the process of collapsing leaves and finite gaps of $\Lc$ can
be performed continuously, so that there is a homotopy between $s_c$ and $p_c$
consisting of Thurston maps.
Moreover, the size of the postcritical set does not change during this homotopy.

We define the point $w\in G$ as the center of $G$, i.e. the unique
point in $G$ that gets eventually mapped to $0$ under $s_c$.
Let $\gamma$ be a simple path that intersects the unit circle exactly once
at the point $\ol\vartheta=\gamma(1/2)$ and such that $\gamma(0)=\infty$, $\gamma(1)=w$.
The Thurston maps $\sigma_\gamma\circ s_c$ and $\sigma_\beta\circ p_c$ are
also combinatorially equivalent.
We can assume that the path homeomorphism $\sigma_\gamma$ maps some
narrow tube $T$ around the curve $\gamma[0,1/2]$ inside the gap $G$.
Consider $T_1=p_c^{-1}(T)$.
This is a strip (a ``tunnel'') connecting two gaps of $\Lc$ outside the unit disk.
Note that the image of $T_1$ under the map $\si_\gamma\circ s_c$ is in the
unit disk, hence is disjoint from $T_1$.
Taking pullbacks of $T_1$ under the iterates of $\si_\gamma\circ s_c$,
we obtain several disjoint tunnels in $\CP^1-\disk$.
It is easy to see that the tunnels are arranged in the same way as the
pullbacks of the critical leaves in the lamination $\Lc_\vartheta^{-1}$,
i.e. the tunnels can be realized as slightly fattened leaves
(we take only finitely many leaves at a time and use that the
tunnels are narrow enough).

We can now formalize the picture with the tunnels.
The mating $p_c\sqcup\Lc_\vartheta$ is modeled by the
union $\Lc\cup\Lc_\vartheta^{-1}$ in the sense that, to obtain the
mating space, we collapse all leaves and finite gaps in this union.
We now modify the ``two-sided lamination'' $\Lc\cup\Lc_\vartheta^{-1}$
in the following way.
The critical leaf $\ell_0^{-1}$ of $\Lc_\vartheta^{-1}$ gets ``fattened'',
i.e. gets transformed into a quadrilateral, whose sides are two geodesic
chords of $\CP^1-\ol\disk$ and two circle arcs (this quadrilateral serves
to model the tunnel $T_1$).
To this end, we need to blow up the endpoints $\ol{-\frac\vartheta 2}$ and
$\ol{-\frac{\vartheta+1}2}$ of $\ell_0^{-1}$ to circle arcs.
We do the same operation with all the pullbacks of $\ell_0$.
As a result, we obtain a geodesic lamination $\Lc_{(\infty)}^{-1}$ in
the complement to the unit disk.
The gaps of $\Lc_{(\infty)}^{-1}$ are {\em ideal quadrilaterals}, whose two sides are
geodesic chords and two other sides are circle arcs, or finite geodesic polygons.
Every leaf of $\Lc_\vartheta^{-1}$ that is not a pullback of $\ell_0$ gives rise to
a leaf of $\Lc_{(\infty)}^{-1}$.
It follows that every finite gap of $\Lc_\vartheta^{-1}$ gives rise to a finite gap
of $\Lc_{(\infty)}^{-1}$.

The process of blowing up certain points of the unit circle into arcs
can be formalized in the following way.
There exists a two-fold orientation-preserving covering $s:\uc\to\uc$
and a monotone continuous projection $\xi:\uc\to\uc$ with the following properties:
\begin{itemize}
 \item
the points $\ol a$ that have non-trivial fibers $\xi^{-1}(\ol a)$
are exactly those with $\si_2^{\circ m}(\ol a)=\ol\vartheta$ for some $m>0$;
\item
the projection $\xi$ semi-conjugates $s$ with $\si_2$, i.e.
$\xi\circ s=\si_2\circ\xi$.
\end{itemize}
Let $K$ be the Cantor set obtained as the closure of the complement in
$\uc$ to all non-trivial fibers of $\xi$.
Then $K$ is invariant under $s$.

With every leaf $\ol{ab}$ of $\Lc_\vartheta$, we associate one or two geodesic chords
of $\CP^1-\ol\disk$.
If the fibers of $\xi$ over $a$ and $b$ are singletons $\{a'\}$ and $\{b'\}$,
respectively, then we associate the chord $\ol{a'b'}\,^{-1}$ with $\ol{ab}$.
Otherwise, we associate two {\em disjoint} chords $\ol{a'b'}\,^{-1}$ and
$\ol{a''b''}\,^{-1}$
with $\ol{ab}$, where $a'$, $a''$ are the endpoints of $\xi^{-1}(a)$,
and $b''$, $b'$ are the endpoints of $\xi^{-1}(b)$.
We can now define a geodesic lamination $\Lc^{-1}_{(\infty)}$ in $\CP^1-\ol\disk$
as the set of all geodesic chords associated with leaves of $\Lc_\vartheta$ in the
way just described.
The lamination $\Lc_{(\infty)}^{-1}$ is $s$-invariant in the sense
of Subsection \ref{ss:inv-lam}.

We can also modify the lamination $\Lc$ so that to make it into an
$s$-invariant lamination $\Lc_{(0)}$.
The lamination $\Lc_{(0)}$ is uniquely defined by the following properties:
\begin{itemize}
 \item
for every leaf $\ol{ab}$ of $\Lc_{(0)}$, the geodesic chord
$\ol{\xi(a)\xi(b)}$ is a leaf of $\Lc$;
\item
let $G^1$, $G^2$ be two gaps of $\Lc$, whose bases map to $G'$ under $\si_2$;
the geodesic convex hull of $\xi^{-1}(G^i\cap\uc)$, $i=1,2$, is a gap of $\Lc_{(0)}$.
\end{itemize}
In other words, as we blow up a point of the unit circle into an arc,
this arc gets inserted to the boundary of an infinite gap of $\Lc$.

Consider the union of $\Lc_{(0)}$ and $\Lc_{(\infty)}$.
It defines an equivalence relation $\approx$ on $\CP^1$: namely, the minimal
equivalence relation containing the equivalence relation generated by $\Lc_{(0)}$
and the equivalence relation generated by $\Lc_{(\infty)}$.
The map $s$ descends to a self-map $g$ of the subset $K/\approx\subset\CP^1/\approx$.
There is a continuous extension of $g$ to all components of the
complement to $K/\approx$ in $\CP^1/\approx$ such that $g$ is a Thurston map.
Components of the complement to $K/\approx$ in $\CP^1/\approx$ will be
referred to as {\em Fatou components} of $g$.
We can arrange that every Fatou component be mapped to a periodic Fatou
component, and that the periodic Fatou components be super-attracting domains.

\begin{thm}
\label{t:topcap}
 The map $g$ defined above is topologically conjugate to the conformal capture
 of $p_c$ corresponding to the capture path $\beta$.
 In particular, the topological space $\CP^1/\approx$ is homeomorphic to
 the sphere.
\end{thm}

The proof of this theorem is organized as follows.
We can extend the map $s$ to a Thurston map such that $s(\ell)=\si_2[\ell]$
for every $\ell\in\Lc_{(0)}$ and $s(\ell^{-1})=\si_2[\ell^{-1}]$
for every $\ell^{-1}\in\Lc_{(\infty)}^{-1}$.
Then it can be shown that $s$ is combinatorially equivalent to $\si_\gamma\circ s_c$,
hence also to $\si_\beta\circ p_c$.
A rough geometric reason for that is the picture with the tunnels discussed above.
Next, we use that a Thurston map combinatorially equivalent to a hyperbolic rational
function is semi-conjugate to this function.
This general theorem proved in \cite{Rees_description} is in fact a
version of Thurston's rigidity principle.
Thus there is a semi-conjugacy between $s$ and a conformal capture.
Finally, the fibers of the semi-conjugacy can be studied, and it can be
proved that the fibers are precisely leaves and finite gaps of the two-sided
lamination $\Lc_{(0)}\cup\Lc_{(\infty)}^{-1}$.

Theorem \ref{t:topcap} has the following immediate corollary.

\begin{cor}
\label{c:mating-sphere}
  The mating space of $p_c\sqcup\Lc_\vartheta$ is homeomorphic to the sphere.
\end{cor}

\begin{proof}
  The mating space is obtained from the space $\CP^1/\approx$ by
  collapsing the images of all ideal quadrilaterals.
  Hence the result follows from the theorem of Moore \cite{Moore} that characterizes
  topological quotients of the sphere that are homeomorphic to the sphere.
\end{proof}

\subsection{Regluing}
\label{ss:top-reg}
Consider a countable set $\Zc$ of disjoint simple curves in the sphere $S^2$.
Recall that $\Zc$ is said to be a {\em null-set} if, for every $\eps>0$,
there exist only finitely many curves from $\Zc$, whose diameter
is bigger than $\eps$.
To measure diameter, we can use any metric compatible with the topology of the sphere.
It is easy to see that the notion of a null-set
does not depend on the choice of a metric.
In fact, the notion of a null-set can be stated in purely topological terms.
Namely, $\Zc$ is a null-set if, for every open covering $\Uc$ of $S^2$,
there exist only finitely many curves from $\Zc$ that are not entirely covered
by an element of $\Uc$.

{\em Regluing data} on $\Zc$ are the choice of an equivalence relation
on each curve $Z\in\Zc$ such that there exists a homeomorphism $h:Z\to [-1,1]$
that transforms this equivalence relation into the equivalence relation
on $[-1,1]$, whose classes are of the form $\{\pm x\}$, $x\in [0,1]$.
To define regluing, we need a null-set of disjoint simple curves $\Zc$ and
a choice of regluing data on them.
We first cut along the curves in $\Zc$, and then reglue these curves in a different way.
To cut along finitely many curves $Z_1$, $\dots$, $Z_n$ means to
consider the Caratheodory compactification of the set $U_n=S^2-\bigcup_{i=1}^n Z_i$,
i.e. the union of the set $U_n$ and the set of all prime ends in $U_n$, equipped with
a suitable topology.
This is a formalization of an intuitively obvious process: as we cut a surface with
boundary along a curve disjoint from the boundary, we obtain a new piece of
the boundary, which is a simple closed curve.
Thus cutting along finitely many disjoint simple curves in the sphere
leads to a compact surface with boundary.
In fact the definition works even in the case, where the curves are not disjoint.

We need to be careful when defining a sphere with countably many cuts.
Suppose that $\Zc$ consists of curves $Z_1$, $\dots$, $Z_n$, $\dots$.
Let $Y_n$ be the result of cutting along the curves $Z_1$, $\dots$, $Z_n$,
i.e. the Caratheodory compactification of $U_n=S^2-\bigcup_{i=1}^n Z_i$.
The natural inclusion $\iota_n:U_{n+1}\to U_n$ gives rise to the continuous map
$\iota_{n*}:Y_{n+1}\to Y_n$ (which is not an inclusion).
Let $Y$ be the inverse limit of the topological spaces $Y_n$ and the continuous
maps $\iota_{n*}$.
The space $Y$ is called {\em the sphere with cuts} (made along the set $\Zc$
of curves).
We will sometimes use the notation $S^2\ominus\Zc$ for $Y$.
In fact, we never used in the definition of $Y$ that curves are disjoint
and that they form a null-sequence.
Thus $Y$ is well defined even without these assumptions.
We will need these assumptions to glue the cuts.
We will also need the regluing data.

Every curve $Z\in\Zc$ gives rise to a simple closed curve $Z^\ominus$ obtained by
cutting along $Z$.
Recall that the regluing data contains an equivalence relation on
$Z$ such that there is a homeomorphism between $Z$ and $[-1,1]$ mapping
every equivalence class onto $\{\pm x\}$ for some $x\in [-1,1]$.
There are two marked points in $Z^\ominus$ that project to the endpoints of $Z$.
We can now define an equivalence relation on $Z^\ominus$ as follows:
two points of $Z^\ominus$ are equivalent if their projections in $Z$ are
equivalent, and they are not separated by the marked points of $Z^\ominus$.
Thus the sphere with cuts $Y$ comes equipped with equivalence relations on
all the cuts.
These equivalence relations extend trivially to an equivalence relation
on the entire space $Y$.
The quotient $Y^*$ of $Y$ by this equivalence relation is called the
{\em regluing} of the sphere $S^2$ along the null-set of disjoint curves
$\Zc$, equipped with regluing data.
We will sometimes use the notation $S^2\#\Zc$ for this regluing.

It can be shown using Moore's characterization \cite{Moore_foundations}
of a topological sphere that
the topological space $Y^*$ is homeomorphic to the sphere (see \cite{T09}).
Thus we reglued a topological sphere and obtained another topological sphere.
This operation becomes useful, however, when we have a geometric structure
on the sphere.
Then the regluing may produce a different geometric structure.

A continuous map $f:S^2\to S^2$ acting on the sphere can be thought of as
a geometric structure.
Geometrically, we can think that the sphere is equipped with arrows connecting
every point $z\in S^2$ with the point $f(z)$.
We now assume that $f$ is an orientation preserving topological branched covering,
and see what happens with arrows when we reglue.
Cutting along a curve creates problems as we cut through the tips of some arrows.
These arrows get doubled, and we obtain two different arrows originating
at the same point.
To rectify this issue, we also need to cut along the pullbacks of the
curve, i.e. along components of its full preimage under the map $f$.
Hence, if we cut along a curve, we also need to cut along all pullbacks.

Suppose that $f:S^2\to S^2$ is a topological branched covering,
and $\alpha_{0}:[0,1]\to S^2$ a simple path such that $\alpha_{0}(0)$ is
a critical value of multiplicity one, and there are no other critical
values in $\alpha_{0}[0,1]$.
Then there is a simple path $\alpha_1:[-1,1]\to S^2$ such that
$f\circ\alpha_1(t)=\alpha_{0}(t^2)$.
The point $\alpha_1(0)$ must be a critical point of $f$.
The path $\alpha_1$ defines the curve $Z_1=\alpha_1[-1,1]$ together with
an equivalence relation on $Z_1$ identifying $\alpha_1(t)$ with $\alpha_1(-t)$
(in other words, two points in $Z_1$ are equivalent if they are mapped to
the same point under $f$).
Let $\Zc$ be the set of all pullbacks of $Z_1$.
We will write $\Zc=[\alpha_0]$.

Suppose that every element $Z\in\Zc$ gets mapped to $Z_1$ one-to-one
under a suitable iterate of $f$.
Suppose also that $\Zc$ is a null-set consisting of disjoint curves.
Under these assumptions, the space $Y^*$ defined above makes sense,
and there is a natural map $F:Y^*\to Y^*$ called the {\em regluing of $f$}.
We started with a topological dynamical system on the sphere, performed
a regluing, and obtained another topological dynamical system on the sphere.
We say that the dynamical system $F:Y^*\to Y^*$ is obtained from
$f:S^2\to S^2$ by regluing the pullbacks of the path $\alpha_0$.
As we frequently think of the map $f:S^2\to S^2$ as a geometric structure
on the sphere, we write $(S^2,f)$ to indicate that $S^2$ is equipped with
this structure, and we sometimes write $(S^2,f)\#\Zc$ or
$(S^2,f)\#[\alpha_0]$ for $(Y^*,F)$,
the space $Y^*$ equipped with the map $F$.

\subsection{Captures vs. regluing}
Let $p_c$ be a quadratic polynomial such that $p_c^{\circ k}(0)=0$,
and $\beta$ a capture path for $p_c$ of external angle $\vartheta$.
We assume that the Thurston map $\sigma_\beta\circ p_c$ is combinatorially
equivalent to a rational function $H$, equivalently, the central gap
of $\Lc_\vartheta(\Omega_c)$ does not disconnect the sphere.
We know by Corollary \ref{c:mating-sphere} that the mating space of
$p_c\sqcup\Lc_\vartheta$ is homeomorphic to the sphere.
Recall that the {\em Julia set} of the mating $p_c\sqcup\Lc_\vartheta$
is defined as the projection of $J_c$ under the quotient map $\pi$ collapsing
all leaves and finite gaps of $\Lc_\vartheta(\Omega_c)$.
The complement to the Julia set is called the {\em Fatou set}.
Connected components of the Fatou set are called {\em Fatou components}.
We will now describe a topological model for the rational function $H$
in terms of the mating $p_c\sqcup\Lc_\vartheta$ and regluing.

Let $V$ be the interior component of $K_c$ containing the point $v=\beta(1)$.
The restriction of $\pi$ to $V$ is a homeomorphism.
Moreover, $\pi(V)$ is a Fatou component of the mating $p_c\sqcup\Lc_\vartheta$
containing the critical value $\pi(b)$ on its boundary, where $b=\beta(1/2)$
as above.
Let $\alpha_0:[0,1]\to\ol{\pi(V)}$ be any simple path such that
$\alpha_0(0)=\pi(b)$ (the critical value of the mating),
$\alpha_0(1)=\pi(v)$ (the {\em center} of the Fatou component $\pi(V)$),
and $\alpha_0(0,1)\subset\pi(V)$.
E.g. we can define $\alpha_0$ as a suitably reparameterized restriction
of $\pi\circ\beta$ to $[1/2,1]$.
Consider the corresponding topological dynamical system $F:Y^*\to Y^*$
obtained by regluing the pullbacks of $\alpha_0$ (up to topological conjugacy,
it does not depend on the choice of $\alpha_0$, provided that $\alpha_0$
satisfies the requirements listed above).

\begin{thm}
  The conformal capture $H$ is topologically conjugate to the regluing $F$ of the
  mating $p_c\sqcup\Lc_\vartheta$.
\end{thm}

The proof is straightforward and can be easily performed by a detailed
comparison of the two topological models.
We only give a sketch here.

\begin{proof}[Sketch of a proof]
  First we replace the capture $H$ with its topological model $g$ described
  in Subsection \ref{ss:topcap}.
  Note that the dynamics of $g$ on the set $K/\approx$ is semi-conjugate
  to the dynamics of the mating $p_c\sqcup\Lc_\vartheta$ on its Julia set.
  Let $h$ denote this semi-conjugacy.
  The fiber $h^{-1}(x)$ of $h$ over any point $x$ in the Julia set of the mating
  is either a singleton or a pair of points.
  Namely, $h^{-1}(x)$ is a pair of points precisely if $x$ gets mapped
  to the critical value $\pi(b)$ of the mating under some strictly positive
  iterate of the mating map.
  Thus the topological model for $g:K/\approx\to K/\approx$ is easy to describe
  in terms of the action of the mating map on its Julia set by doubling certain
  points.
  Note that the regluing does exactly the same thing with the dynamics on
  the Julia set.
  Thus $g:K/\approx\to K/\approx$ is topologically conjugate to the
  restriction of the regluing to the image of the Julia set under regluing.
  It remains to extend the topological conjugacy over all
  components of $\CP^1/\approx-K/\approx$, which is straightforward.
\end{proof}

Thus the map $F$ obtained from $p_c\sqcup\Lc_\vartheta$ by regluing
is a topological model for the capture.
Note that many different paths give rise to the same capture.
Hence we obtain many topological models for the same capture.
It is not at all obvious that these models are topologically conjugate
unless we use that they are all conjugate to the capture.
We would be interested to know a direct proof of the topological conjugacy
between the models.

\subsection{Reversed regluing}
The operation of regluing is reversible.
Consider a null-set $\Zc$ of simple disjoint curves in $S^2$, equipped with
regluing data.
Regluing of this set yields a topological space $Y^*=S^2\#\Zc$ homeomorphic
to the sphere.
Actually, it yields more than that.
We also obtain a null-set $\Zc^\#$ of simple disjoint curves
in the space $Y^*$ equipped with regluing data.
Namely, we define $\Zc^\#$ as the set of images of all cuts under the natural
projection from $Y$ to $Y^*$.
Thus the curves in $\Zc^\#$ are in one-to-one correspondence with the curves
in $\Zc$.
Given any $Z\in\Zc$, we first cut it to obtain a simple closed curve $Z^\ominus$
and then glue it back in a different way to obtain the corresponding curve
$Z^\#\in\Zc^\#$.
There are natural projections from $Z^\ominus$ to $Z$ and to $Z^\#$.
We can now define an equivalence relation on $Z^\#$ in terms of these projections.
Namely, two points of $Z^\#$ are equivalent if they are projections of
points, whose images in $Z$ coincide.
These equivalence relations define regluing data on the set of curves $\Zc^\#$.
In particular, we can consider the regluing $(S^2\#\Zc)\#\Zc^\#$.
This topological space is {\em canonically} homeomorphic to $S^2$.

Suppose now that the sphere $S^2$ is equipped with a topological
branched covering $f:S^2\to S^2$, and $\Zc=[\alpha_0]$, where a
simple path $\alpha_0$ is as in Subsection \ref{ss:top-reg}.
Then, after regluing, we have $\Zc^\#=[\alpha_0^\#]$, where
the path $\alpha^\#_0$ is obtained as the image of the path
$\alpha_0$ in the space $S^2\#\Zc$ (the multivalued correspondence
between points of $S^2$ and points of $S^2\#\Zc$ is in fact single
valued on $\alpha_0[0,1]$ since this set is disjoint from all the cuts).
We have
$$
(S^2,f)=((S^2,f)\#[\alpha_0])\#[\alpha^\#_0].
$$
This means that, to recover $(S^2,f)$ from the regluing $(S^2,f)\#\Zc$,
we only need to reglue the set of curves $[\alpha_0^\#]$.

\section{Captures and matings in parameter slices}
We now consider some natural complex one-dimensional parameter spaces of
rational functions, and discuss parameter values that correspond to
captures and matings.
Recall now that any rational function $f:\CP^1\to\CP^1$ of degree
at least two defines the splitting of the Riemann sphere into the
Fatou set and the Julia set.
However, these sets are not defined in the same way as for polynomials, because
infinity is no better than any other point of the sphere, when a non-polynomial
rational function acts.
By definition, the Fatou set of $f$ consists of all points, at
which the map $f$ is Lyapunov stable, equivalently, the sequence $f^{\circ n}$
is equicontinuous.
The Julia set is by definition the complement to the Fatou set.

\subsection{The slices $Per_k(0)$}
Consider the space $Rat_2$ of conformal conjugacy classes of
all quadratic rational functions with marked critical points.
This space is complex two-dimensional: a quadratic rational function is
determined by five coefficients, and the group $\Aut(\CP^1)$
of conformal automorphisms of the Riemann sphere has complex dimension three.
Thus the dimension of the space $Rat_2$ equals the number of critical
points a quadratic rational function has, and this is not a mere coincidence.
It is a general observation in holomorphic dynamics that the behavior of
critical orbits is to a large extent responsible for the dynamical behavior
of the whole map.
There are a number of theorems to this effect saying roughly that if
all critical orbits behave nicely, then the function itself is nice
from the dynamical viewpoint.
For example, if all critical orbits get attracted by attracting cycles, then
the function is {\em hyperbolic}, i.e. 1) there exists a neighborhood
of the Julia set and a Riemannian metric on this neighborhood such that
the function is strictly expanding with respect to this metric, 2)
there exists a Riemannian metric on the Fatou set, in which the function
is strictly contracting.

To simplify the model (and to make nice pictures) people consider
complex one-dimensional slices of $Rat_2$.
To define the slices, one fixes a particular nice behavior of one critical
point, so there remains only one ``free'' critical point.
E.g. one can impose that one critical point is periodic of period $k$.
We will denote the corresponding slice by $Per_k(0)$, following J. Milnor
\cite{Milnor-QuadRat} (0 in the notation stands for the multiplier of
a $k$-periodic point: having a periodic point of multiplier 0 is the same
as having a periodic critical point).
More precisely, the space $Per_k(0)$ is defined as the set of all conformal
conjugacy classes of rational functions $f$ with marked critical points $c_1$,
$c_2$ such that $f^{\circ k}(c_1)=c_1$, and $k$ is minimal with this property.
Clearly, each $Per_k(0)$ is an algebraic curve in the algebraic surface $Rat_2$.
For $k=1$, 2, 3 and 4, the genus of this curve is equal to zero, i.e.
there is a rational parameterization.

If $k=1$, then one critical point must be fixed.
By a conformal coordinate change, we can map this point to infinity.
A rational function, for which the infinity is a fixed critical point,
is necessarily a quadratic polynomial.
By an affine change of variables, the coefficient with $z$ can be killed,
so that every quadratic polynomial reduces to the form $z^2+c$.
Thus $Per_1(0)$ can be identified with the standard quadratic family $\{z^2+c\}$.

Consider the case $k=2$.
Any conjugacy class from $Per_2(0)$ that does not contain the map $z\mapsto 1/z^2$
has a unique representative of the form $z\mapsto a/(z^2+2z)$.
Thus $Per_2(0)$ can be identified with the $a$-plane punctured at $0$,
to which we need to add a single point at infinity corresponding to the class of
the map $1/z^2$.
Rational parameterizations for $Per_3(0)$ and $Per_4(0)$ are also easy to obtain.

\subsection{Hyperbolic components}
The set of elements in $Per_k(0)$ representing hyperbolic functions is open.
Connected components of this set are called {\em hyperbolic components}.
J. Milnor \cite{Milnor-QuadRat} gave a classification of hyperbolic components
in $Per_k(0)$ into four types: A, B, C and D.
Hyperbolic elements in $Per_k(0)$ of type A are classes of rational functions
such that both critical points $c_2$ and $c_1$ are in the same
super-attracting domain (i.e. $c_2$ lies in the Fatou component containing $c_1$).
It can be proved that there are no type $A$ components for $k>1$,
and the space $Per_1(0)$ has just one type A component that is identified with
the complement to the Mandelbrot set.
Hyperbolic elements in $Per_k(0)$ of type B are classes of rational functions
such that the free critical point $c_2$ lies in a periodic Fatou component,
whose cycle contains $c_1$ but which itself does not contain $c_1$.
Every slice $Per_k(0)$ contains a finite number of type B hyperbolic components
(this number is nonzero unless $k=1$).
Hyperbolic elements of $Per_k(0)$ of type C are classes of rational functions
such that the free critical point $c_2$ lies in a strictly preperiodic Fatou
component that gets eventually mapped to the component containing $c_1$.
For all $k>1$, the slices $Per_k(0)$ contain infinitely many hyperbolic type C
components.
Finally, hyperbolic elements of $Per_k(0)$ of type D are classes of rational
functions such that $c_1$ and $c_2$ lie in disjoint periodic cycles of Fatou
components.
All slices $Per_k(0)$ contain infinitely many hyperbolic type D components.

It follows from \cite{McMullen}
that every hyperbolic component in $Per_k(0)$
of type B, C or D has a unique {\em center}, i.e. a critically finite conjugacy class.
A conformal mating of two hyperbolic critically finite quadratic polynomials 
represents the center of some type D hyperbolic component in $Per_k(0)$.
Similarly, a conformal capture of a hyperbolic critically finite quadratic polynomial
represents the center of some type C hyperbolic component.
However, the converse is not true in general.
It is true for $Per_2(0)$ but, in the slice $Per_3(0)$,
there are type D components, whose centers are not matings, and there are
type C components, whose centers are not captures.
Examples are given in \cite[Appendix F by J. Milnor and Tan Lei]{Milnor-QuadRat},
\cite{ReesV3}, respectively.
We say that a hyperbolic component of type C is a {\em capture component}
if its center is a conformal capture.

\subsection{Regluing and type C boundaries}
In \cite{T09}, topological models were given for classes in $Per_k(0)$
that lie on the boundaries of type C hyperbolic components.
These models were defined in terms of regluing.
We now cite the result:

\begin{thm}
\label{t:regl}
 Suppose that the class of a rational function $f$ belongs to
 the boundary of a type C hyperbolic component $\Hc$ in $Per_k(0)$
 but does not belong to the boundary of a type B component.
 Then $f$ is topologically conjugate to a map obtained from the
 center of $\Hc$ by regluing.
\end{thm}

We now make this statement more precise.
Suppose that a rational function $f$ represents a point in the parameter
slice $Per_k(0)$ lying on the boundary of $\Hc$ but not lying on the
boundary of a type B hyperbolic component.
Then it can be shown (\cite[Subsection 3.3]{T09}, which imitates an
argument from \cite{A-Y}) that the critical point $c_2$ of $f$
belongs to the boundary of some Fatou component $W$ that gets eventually mapped
to a Fatou component containing $c_1$.
The critical value $f(c_2)$ belongs to the boundary of $f(W)$.
There are two components of $f^{-1}(f(W))$, say, $W$ and $\tilde W$.
The boundary of each of the two components contains the critical point $c_2$.
It may also happen that there are more than two Fatou components,
whose boundary contains $c_2$.
Equivalently, there may be more than one Fatou component, whose boundary
contains $f(c_2)$.
However, the choice of the component $f(W)$ is determined by the
choice of a hyperbolic component $\Hc$, as is explained in
\cite[Subsection 3.3]{T09}.

Let $m$ be the minimal positive integer such that $f^{\circ m}(W)\ni c_1$.
The Fatou component $W$ has a unique center, i.e. a point $z_W$ such
that $f^{\circ m}(z_W)=c_1$.
We now consider a simple path $\alpha_0:[0,1]\to\ol{f(W)}$ such that
$\alpha_0(0)=f(c_2)$, $\alpha_0(1)=f(z_W)$, and $\alpha_0(0,1)\subset f(W)$.
Such a path exists because the boundary of $f(W)$ is locally connected
(see \cite{A-Y,T09}).
There is a unique simple path $\alpha_1:[-1,1]\to\CP^1$ such that
$f\circ\alpha_1=\alpha_0$.
Then $\alpha_1(0)$ must coincide with the critical point $c_2$.
We can now reglue the path $\alpha_1$ and all its pullbacks according to
the construction given above.
Let $H$ be a hyperbolic critically finite rational function representing
the center of $\Hc$.
Then $(\CP^1,H)$ is topologically conjugate to $(\CP^1,f)\#[\alpha_0]$.
Since regluing is reversible, we can also obtain the topological
dynamical system $(\CP^1,f)$ as a regluing $(\CP^1,H)\#[\alpha^\#_0]$.
Here the simple path $\alpha^\#_0:[0,1]\to\CP^1$ connects the
non-periodic critical value of $H$ to a boundary point of the Fatou component
containing it.
This statement is a more precise form of Theorem \ref{t:regl} cited above.

\subsection{Angles on the boundaries of capture components}
Let $\Hc$ be a capture component, i.e. a type C hyperbolic component in $Per_k(0)$,
whose center is a capture.
Every point of $\d\Hc$ that is not on the boundary of a type B component
is determined by its {\em angle}.
Let $f$ be a rational function representing this point of the parameter slice.
Then the critical value $f(c_2)$ of $f$ belongs to the boundary of
a Fatou component $f(W)$ that gets eventually mapped to the Fatou component
containing $c_1$.
Moreover, as was mentioned above, the choice of the Fatou component $f(W)$
is determined by the choice of the type C component $\Hc$, whose boundary contains
the class of $f$.

{\em Angles} of points on the boundary of $f(W)$ are defined similarly
to internal angles in polynomial case.
By the B\"ottcher theorem, there exists a bi-holomorphic map
$\psi:\disk\to f^{\circ m}(W)$ that conjugates the map
$z\mapsto z^2$ with the restriction
of the map $f^{\circ k}$ to $f^{\circ m}(W)$.
Since the boundary of $f^{\circ m}(W)$ is locally connected, there is
a continuous extension $\ol\psi:\ol\disk\to\ol{f^{\circ m}(W)}$.
The point of angle $\varkappa\in\R/\Z$ on the boundary of $f^{\circ m}(W)$
is by definition the point $\ol\psi(\ol\varkappa)$.
The point of angle $\varkappa\in\R/\Z$ on the boundary of $f(W)$ is
by definition the point $z\in\d f(W)$ such that $f^{\circ m-1}(z)$ is
the point of angle $\varkappa$ on the boundary of $f^{\circ m}(W)$.
Since the boundary of $f(W)$ maps one-to-one onto the boundary
of $f^{\circ m}(W)$ under $f^{\circ m-1}$, the point of angle $\varkappa$
on the boundary of $f(W)$ is well defined.

Recall that a set $A_\lambda\subset\CP^1$ depending on a parameter
$\lambda$ (taking values in a Riemann surface $\Lambda$)
{\em moves holomorphically} with $\lambda$
if there is a subset $A\subset\CP^1$ and a map
$(a,\lambda)\mapsto \iota_\lambda(a)$ (a {\em holomorphic motion})
from $A\times\Lambda$ to $\CP^1$ that is holomorphic with respect to $\lambda$
for every fixed $a\in A$, injective with respect to $a$
for every fixed $\lambda\in\Lambda$,
and such that $\iota_\lambda(A)=A_\lambda$ for every $\lambda\in\Lambda$.
A theorem, sometimes called the {\em $\lambda$-lemma},
of Ma\~ne, Sud and Sullivan \cite{MSS}
claims that if $A_\lambda$ moves holomorphically with $\lambda$, and
$\iota_{\lambda_0}$ is a quasi-symmetric embedding for some $\lambda_0\in\Lambda$,
then all $\iota_\lambda$ are quasi-symmetric embeddings; moreover,
the closure $\ol A_\lambda$ also moves holomorphically with $\lambda$.

Let $\lambda\in Per_k(0)$ be a parameter value, and $f_\lambda$ a
rational function representing $\lambda$.
Suppose that $f=f_{\lambda_0}$.
At least for the values of $\lambda$ that are close to $\lambda_0$,
we can choose representatives so that $f_\lambda$ depends holomorphically on $\lambda$.
There is a holomorphic motion that includes the sets $V_\lambda=f_\lambda(W_\lambda)$
such that $V_{\lambda_0}=f(W)$, and $V_\lambda$ is a Fatou component of $f_\lambda$
(see the proof of Proposition \ref{p:angle} that follows for more detail
on this holomorphic motion).
By the $\lambda$-lemma, the boundaries $\d V_\lambda$ also move holomorphically.
We can continue this holomorphic motion all the way up to the center $\lambda_1$
of the hyperbolic component $\Hc$.
As follows from the topological model for captures, the boundary of
$f_{\lambda_1}(W_{\lambda_1})$ is homeomorphic to the boundary of
a periodic Fatou component in the mating $p_c\sqcup\Lc_\vartheta$.
On the other hand, by the holomorphic motion argument, the boundary of
$f_{\lambda_1}(W_{\lambda_1})$ is homeomorphic to the boundary of $f(W)$.
Therefore, by Proposition \ref{p:A}, the boundary of $f(W)$ is a Jordan curve.
It follows that different angles cannot correspond to the same point on
the boundary of $f(W)$, i.e. every point on the boundary of $f(W)$ has a
well-defined angle.

We can now define the angle of the point in $\d\Hc$ represented by $f$
as the angle of the critical value $f(c_2)$ in $\d f(W)$ (in the sense just described).
If $\Bc$ denotes the union of all type B components, then the angle
is a function on $\d\Hc-\ol\Bc$ with values in $\R/\Z$.

\begin{prop}
\label{p:angle}
  The angle is an injective continuous function on $\d\Hc-\ol\Bc$.
\end{prop}

We need the following simple lemma (cf. also \cite{A-Y}):

\begin{lem}
\label{l:hol-mot}
  Suppose that $z_\lambda\in\CP^1$ is a point, and $R_\lambda\subset\CP^1$ is a
  set moving holomorphically with $\lambda\in\Lambda$.
  If $z_{\lambda_0}\in \ol R_{\lambda_0}$ for some $\lambda_0\in\Lambda$, then we
  either have $z_\lambda\in R_\lambda$ for values of $\lambda$ arbitrarily close
  to $\lambda_0$, or $z_\lambda\in\ol R_\lambda-R_\lambda$ for all $\lambda$
  sufficiently close to $\lambda_0$.
\end{lem}

\begin{proof}
  Let $\iota:R\times\Lambda\to\CP^1$ be the holomorphic motion of $R_\lambda$ so that
  $R_\lambda=\iota_\lambda(R)$ for all $\lambda\in\Lambda$.
  We assume that $R\subset\CP^1$ and $\iota_{\lambda_0}$ is a homeomorphic
  embedding.
  By the $\lambda$-lemma, $\iota$ extends to a holomorphic motion
  $\ol\iota:\ol R\times \Lambda\to\CP^1$.
  Consider the holomorphic functions $z_\lambda$ and
  $w_\lambda=\ol\iota_\lambda(r_0)$ of $\lambda$,
  where $r_0\in\ol R$ is the point such that $\ol\iota_{\lambda_0}(r_0)=z_{\lambda_0}$.
  By definition, the holomorphic function $z_\lambda-w_\lambda$ of $\lambda$
  vanishes at $\lambda=\lambda_0$.
  As $\lambda$ goes around the circle $|\lambda-\lambda_0|=\eps$, the point
  $z_\lambda-w_\lambda$ makes at least one loop around $0$, unless
  $z_\lambda=w_\lambda$ identically for $\lambda$ in some neighborhood of $\lambda_0$.
  Therefore, for $r\in R$ very close to $r_0$, the point $z_\lambda-\iota_\lambda(r)$
  also makes at least one loop around 0.
  We conclude that the function $z_\lambda-\iota_\lambda(r)$ vanishes for some
  $\lambda$ in the disk $|\lambda-\lambda_0|<\eps$.
\end{proof}

\begin{proof}[Proof of Proposition \ref{p:angle}]
  Let $\lambda\in Per_k(0)$ be a class of rational functions, and $f_\lambda$
  a rational function representing $\lambda$.
  We can choose representatives $f_\lambda$, at least locally, so that
  that they depend holomorphically on $\lambda$.
  We let $A_\lambda$ denote the immediate basin of the periodic critical point
  $c_1^\lambda$ of $f_\lambda$, i.e. the Fatou component of $f_\lambda$
  containing $c_1^\lambda$.
  We know that $A_\lambda$ moves holomorphically with $\lambda$.
  More precisely, the bi-holomorphic isomorphism $\psi_\lambda:\disk\to A_\lambda$
  conjugating the map $z\mapsto z^2$ with the restriction of
  $f_\lambda^{\circ k}$ to $A_\lambda$ is a holomorphic motion.

  Define the {\em ray of angle $\theta$} in $A_\lambda$ as the set of all points
  $R_\lambda(t,\theta)=\psi_\lambda(e^{-t+2\pi i\theta})$, where $t$ runs through
  $(0,\infty)$ (this ray will sometimes be denoted by $R_\lambda(\theta)$).
  If $f_\lambda$ is hyperbolic, then every ray {\em lands}, i.e. there exists
  a limit of $R_\lambda(t,\theta)$ as $t\to 0$.
  By the $\lambda$-lemma, the closure $\ol{R_\lambda(\theta)}$
  of every ray moves holomorphically with $\lambda\in\Lambda=Per_k(0)-\ol\Bc$.
  We denote the corresponding holomorphic motion by $\ol\psi_\lambda$.
  It follows that the ray $R_\lambda(\theta)$ always lands at the point
  $\ol\psi_\lambda(\ol\theta)$.

  Suppose that, for every $\lambda\in\Hc$, the non-periodic critical point
  $c^\lambda_2$ belongs to some Fatou component $W_\lambda$ such that
  $f_\lambda^{\circ m}(W_\lambda)=A_\lambda$,
  and $m$ is the smallest positive integer with this property
  (clearly, the number $m$ does not depend on the choice of $\lambda$ in $\Hc$).
  Consider the map $\lambda\mapsto \psi_\lambda^{-1}(f_\lambda^{\circ m}(c^\lambda_2))$.
  This is a holomorphic map from $\Hc$ to $\disk$.
  It is proved in \cite{Rees_components} that this map is actually a
  bi-holomorphic isomorphism between $\Hc$ and $\disk$.
  Let $\Psi_{\Hc}$ denote the inverse of this map.

  Define the {\em parameter ray} $\Rc_{\Hc}(\theta)$ in $\Hc$ as the set of
  all points of the form $\Rc_{\Hc}(t,\theta)=\Psi_{\Hc}(e^{-t+2\pi i\theta})$,
  where $t\in (0,\infty)$.
  We will now prove that every point $\lambda_0\in\d\Hc-\ol{\Bc}$ is the landing
  point of exactly one parameter ray.
  Indeed, the non-periodic critical point $c_2^{\lambda_0}$ of $f_{\lambda_0}$ has a
  well-defined angle $\theta$, i.e.
  $f_{\lambda_0}^{\circ m}(c_2^{\lambda_0})\in\ol{R_{\lambda_0}(\theta)}$.
  Both the point $f_\lambda^{\circ m}(c_2^\lambda)$ and the set
  $\ol{R_\lambda(\theta)}$ move holomorphically with $\lambda$.
  By Lemma \ref{l:hol-mot}, the point $\lambda_0$ is an accumulation point
  of the parameter ray $\Rc_{\Hc}(\theta)$, i.e. a partial limit of
  $\Rc_{\Hc}(t,\theta)$ as $t\to 0$.
  A point $\lambda_0$ in $\d\Hc-\ol\Bc$ of angle $\theta$ is
  a zero of the holomorphic function
  $$
  \lambda\mapsto f_\lambda^{\circ m}(c_2^\lambda)-\ol\psi_\lambda(e^{2\pi i\theta}).
  $$
  Since this function is not constant, its zeros must be isolated.
  Therefore, there are no other points of angle $\theta$
  in a neighborhood of $\lambda_0$ in $\Lambda$.
  On the other hand, by \cite[Proposition 5]{T09}, all accumulation
  points of the parameter ray $\Rc_{\Hc}(\theta)$ must have angle $\theta$.
  Since the set of accumulation points is connected, the ray $\Rc_{\Hc}(\theta)$
  must land at $\lambda_0$.
  Since any parameter ray can only land at one point, we obtain
  that the angle is an injective function on $\d\Hc-\ol\Bc$.
  The continuity of this function follows from the fact that a zero
  of a holomorphic function depending continuously on parameters
  moves continuously with respect to parameters.
\end{proof}

\begin{cor}
  Suppose that $\Hc$ is a capture hyperbolic component, whose closure is
  disjoint from $\ol\Bc$.
  Then the angles establish a homeomorphism between $\d\Hc$ and $\R/\Z$.
\end{cor}

\subsection{Identification of matings}
As above, let $\Hc$ be a capture hyperbolic component in $Per_k(0)$ with
center $\lambda_1$ (we will write $H=f_{\lambda_1}$ for the corresponding
capture), and $f=f_{\lambda_0}$ a map representing a point $\lambda_0$
on the boundary of $\Hc$.
Suppose that this point is not on the boundary of a type B
component, and suppose also that this point has angle $\varkappa$ in $\d\Hc$.
We know that the dynamical system $(\CP^1,f)$ is obtained from
the dynamical system $(\CP^1,H)$ by regluing the system $[\alpha_0^\#]$
of curves, where the simple path $\alpha_0^\#$ connects the non-periodic
critical value $H(c^{\lambda_1}_2)$ of $H$ lying in some Fatou
component $H(W_{\lambda_1})$ with the boundary point of this Fatou component
of angle $\varkappa$.

On the other hand, $(\CP^1,H)$ is obtained by regluing from the
topological mating $p_c\sqcup\Lc_\vartheta$ for various values of $\vartheta$.
The capture path determines both an external angle $\vartheta$ and
and the internal angle $\varkappa$ of the point $b=\beta(1/2)$.
We will say that $\beta$ is a {\em capture path of internal angle} $\varkappa$.
Note however, that, in contrast to external angle, knowing the
internal angle and the Fatou component of $p_c$ containing
the point $v=\beta(1)$ may not uniquely determine a capture path.
Suppose that the angle $\varkappa$ is equal to the
angle of the point $\lambda_0\in\d\Hc$.
Then the path that was used to reglue the mating into $(\CP^1,H)$
was a suitably reparameterized restriction $\beta:[1/2,1]\to K_c$.
The corresponding path in $(\CP^1,H)$ that appears after the regluing
is a simple path connecting the critical value $H(c^{\lambda_1}_2)$ with the
point on the boundary of $H(W_{\lambda_1})$ with angle $\varkappa$.
Therefore, the reverse regluing of $(\CP^1,H)$ into the mating
$p_c\sqcup\Lc_\vartheta$ is the same as the regluing of $(\CP^1,H)$
into $(\CP^1,f)$!
It follows that the results should also be the same, up to topological conjugacy.
We have thus proved the following

\begin{thm}
  \label{t:same-reg}
  Let $p_c$ be a quadratic polynomial such that the critical point 0 of $p_c$
  is periodic of minimal period $k$, and $\beta:[0,1]\to\CP^1$ a capture path for $p_c$
  of external angle $\vartheta$ and internal angle $\varkappa$.
  Suppose that $\sigma_\beta\circ p_c$ is combinatorially equivalent to
  a rational function $H$.
  If $\Hc$ is the hyperbolic component in $Per_k(0)$, whose center is
  represented by $H$, and a rational function $f$ represents a boundary
  point of $\Hc$ of angle $\varkappa$ not lying on the boundary of a type B
  hyperbolic component, then $f$ is topologically conjugate to the
  mating $p_c\sqcup\Lc_\vartheta$.
\end{thm}

We can make this theorem more precise.

\begin{prop}
\label{p:lam-poly}
  The lamination $\Lc_\vartheta^c$ models some quadratic polynomial.
\end{prop}

\begin{proof}
 The statement will follow if we prove that, in the parameter plane
 of complex polynomials, the external ray of angle $\vartheta$ lands at
 a unique point of the Mandelbrot set, and  the quadratic polynomial
 corresponding to this point has locally connected Julia set.
 By the theorem of Yoccoz on local connectivity (see e.g. \cite{Milnor-LC}),
 actually, by a simple version of it, this is true if the critical leaf or
 gap $g$ of $\Lc^c_\vartheta$ is non-recurrent, i.e. no iterated $\si_2$-image
 of $g'$ intersects a small neighborhood of $g'$.

 If $\Lc_\vartheta^c$ has a critical gap rather than a critical leaf,
 then this critical gap must be preperiodic, hence it is not recurrent.
 We now assume that $\Lc_\vartheta^c$ has a critical leaf $\ell_0$, hence
 $\Lc_\vartheta^c=\Lc_\vartheta$.
 Let $\Lc$ be the quadratic invariant lamination that models the polynomial $p_c$.
 We will write $G^1$ and $G^2$ for the strictly preperiodic gaps of $\Lc$ that
 contain the endpoints of $\ell_0^*$.
 The bases of the gaps $G^1$ and $G^2$ map to the basis of a strictly preperiodic
 gap $G$ of $\Lc$.
 If $\ell_0^*$ were recurrent, then the gap $G$ contained a recurrent
 point or a recurrent edge.
 Since the forward orbit of $G$ contains only finitely many gaps and since
 two intersecting gaps have always an edge in common,
 there are no recurrent points in $G$ that are not in edges of $G$.
 Every edge of $G$ is eventually periodic.
 However, it follows from Lemma \ref{l:str-per} that all edges are strictly
 preperiodic, hence non-recurrent.
\end{proof}

It follows that the conclusion of Theorem \ref{t:same-reg} can be made stronger:
{\em the rational function $f$ is topologically conjugate to the mating
of $p_c$ with some quadratic polynomial!}
The next question is: given a capture hyperbolic component $\Hc$ in $Per_k(0)$,
how many of the boundary points of $\Hc$ correspond to matings?
We will see that, in some cases, all boundary maps correspond to
matings, and in some cases, there is a simple arc on the boundary of
$\Hc$ consisting of matings.

\subsection{End-captures and cut-captures}
We now discuss how much a capture depends on the choice of a capture path.
Let $p_c$ be a quadratic polynomial such that $0$ is periodic of minimal
period $k$, and $\beta:[0,1]\to\CP^1$ a capture path for $p_c$.
Note that the combinatorial class of the Thurston map $\sigma_\beta\circ p_c$
depends only on the homotopy class of the path $\beta$ relative to
the forward orbit of the point $v=\beta(1)$ (this forward orbit is a finite
set by definition of a capture path).
Let $V$ be the Fatou component of $p_c$ containing the point $v$.
Define {\em limbs} of $V$ as the closures of the complement to $\ol V$
in the filled Julia set $K_c$.
As was noted in \cite{ReesV3}, the iterated forward images of $v$ under $p_c$
are contained in only one or two limbs of $V$.
In the first case, we say that $\beta$ is an {\em end-capture path}.
In the second case, we say that $\beta$ is a {\em cut-capture path}.
A rational function (if any) combinatorially equivalent to $\sigma_\beta\circ p_c$
is called an {\em end-capture} or a {\em cut-capture} according to
whether $\beta$ is an end-capture path or a cut-capture path.
Let us first consider a hyperbolic component in $Per_k(0)$ corresponding to
an end-capture.

\begin{thm}
  \label{t:end-cap}
  Let $\Hc$ be a hyperbolic component in $Per_k(0)$, whose center
  is represented by an end-capture $H$ of $p_c$.
  Then representatives of all points in $\d\Hc-\ol\Bc$ are matings of
  $p_c$ with certain quadratic polynomials.
\end{thm}

\begin{proof}
  Let $\lambda_0$ be a point in $\d\Hc-\ol\Bc$ of angle $\varkappa$, and $f$
  a rational function representing $\lambda_0$.
  Suppose that $H$ is the capture of $p_c$ corresponding to some capture
  path $\beta_0$.
  Consider any capture path $\beta$ for $p_c$ such that $\beta(1)=\beta_0(1)$,
  and the point $b=\beta(1/2)$ has internal angle $\varkappa$ with respect to the
  Fatou component of $p_c$ containing the point $\beta_0(1)$.
  Since $H$ is an end-capture, it is also the capture of $p_c$ corresponding
  to the capture path $\beta$.
  Theorem \ref{t:same-reg} and Proposition \ref{p:lam-poly} now imply that $f$
  is topologically conjugate to the mating of $p_c$ with some quadratic polynomial.
\end{proof}

We now consider cut-capture paths for $p_c$.
Let $V$ be a strictly preperiodic Fatou component of $p_c$.
Suppose that the forward orbit of $V$ is contained in two limbs of $V$.
In this case, there are two homotopy classes of capture paths terminating in $V$.
One class $\Cc_0$ contains capture paths, whose internal angles are in $(0,1/2)$,
and the other class $\Cc_1$ contains capture paths, whose internal
angles are in $(1/2,1)$.
Capture paths with internal angles $0$ and $1/2$ can belong to either class.
There are, in general, two conformal captures $H_0$ and $H_1$ of $p_c$, up to conformal conjugacy.
For every capture path in $\Cc_0$, the corresponding capture is $H_0$,
and for every capture path in $\Cc_1$, the corresponding capture is $H_1$.
Let $\Hc_0$ and $\Hc_1$ denote the hyperbolic components of $Per_k(0)$,
whose centers are represented by $H_0$ and $H_1$, respectively.
Since $H_0$ corresponds to capture paths with internal angles from 0 to $1/2$,
we call $\Hc_0$ a $[0,1/2]$-capture component associated with $p_c$.
Similarly, we call $\Hc_1$ a $[1/2,1]$-capture component.

\begin{thm}
  \label{t:cut-cap}
  Let $\Hc_0$ be a $[0,1/2]$-capture hyperbolic component in $Per_k(0)$
  associated with $p_c$.
  Then any point of $\d\Hc_0-\ol\Bc$, whose angle belongs to $[0,1/2]$,
  is represented by a mating of $p_c$ with some quadratic polynomial.
  Similarly, let $\Hc_1$ be a $[1/2,1]$-capture hyperbolic component in
  $Per_k(0)$ associated with $p_c$.
  Then any point of $\d\Hc_1-\ol\Bc$, whose angle belongs to $[1/2,1]$,
  is represented by a mating of $p_c$ with some quadratic polynomial.
\end{thm}

This theorem leads to the following question: {\em is it true that
every $[0,1/2]$-capture component is simultaneously a $[1/2,1]$-capture component?}
If this is true, then, for every hyperbolic component $\Hc$ in $Per_k(0)$,
all points of $\d\Hc-\ol\Bc$ are represented by matings.
Conjecturally, {\em all points of $\d\Hc$ are represented by matings}, including
those in $\ol\Bc$.


\begin{thebibliography}{999}

\bibitem[AY]{A-Y}
M. Aspenberg, M. Yampolsky, ``Mating non-renormalizable quadratic polynomials'',
Commun. Math. Phys., \textbf{287} (2009), p.1--40

\bibitem[DH]{DH}
A. Douady and J. Hubbard,
``A proof of Thurston's topological characterization of rational functions'',
Acta Math. \textbf{171} (1993), 263--297


\bibitem[Mc]{McMullen}
C. McMullen, ``Automorphisms of rational maps'',
In ``Holomorphic Functions and Moduli I'', pages 31--60, Springer, 1988.

\bibitem[Mo16]{Moore_foundations}
R.L. Moore, ``On the foundations of plane analysis situs'',
Transactions of the AMS, \textbf{17} (1916), 131--164

\bibitem[Mo25]{Moore}
R.L. Moore, ``Concerning upper-semicontinuous collections of continua'',
Transactions of the AMS, \textbf{27}, Vol. 4 (1925), 416--428

\bibitem[MSS]{MSS}
R. Ma\~ne, P. Sud, D. Sullivan, ``On the dynamics of rational maps'',
Ann. Sci. \'Ecole Norm. Sup. (4) \textbf{16} (1983), no. 2, 193--217.

\bibitem[M93]{Milnor-QuadRat}
J. Milnor, ``Geometry and Dynamics of Quadratic Rational Maps''
Experimental Math. \textbf{2} (1993) 37--83

\bibitem[M00a]{Milnor-Mand}
J. Milnor, ``Periodic orbits, externals rays and the Mandelbrot set: an expository account''.
G\'eom\'etrie complexe et syst\`emes dynamiques (Orsay, 1995).
Ast\'erisque No. \textbf{261} (2000), 277--333

\bibitem[M00b]{Milnor-LC}
J. Milnor, ``Local connectivity of Julia sets: expository lectures'', in
``The Mandelbrot set, Theme and Variations,''
LMS Lecture Note Series \textbf{274}, Cambr. U. Press (2000), 67--116

\bibitem[R90]{Rees_components}
M. Rees, ``Components of degree two hyperbolic rational maps''
Invent. Math. \textbf{100} (1990), 357--382

\bibitem[R92]{Rees_description}
M. Rees, ``A partial description of the Parameter Space of Rational Maps of Degree
Two: Part 1''
Acta Math. \textbf{168} (1992), 11--87

\bibitem[R10]{ReesV3}
M. Rees, ``A Fundamental Domain for $V_3$'',
M\'em. Soc. Math. Fr. (N.S.) No. \textbf{121} (2010)

\bibitem[Su]{Sul}
D. Sullivan, `` Quasiconformal homeomorphisms and dynamics. I. Solution of the Fatou-Julia problem on wandering domains''. Ann. of Math. (2) \textbf{122} (1985), No. 3, 401--418

\bibitem[TL]{TanLei}
L. Tan, ``Matings of quadratic polynomials'', Erg. Th. and Dyn. Sys. \textbf{12} (1992) 589--620

\bibitem[Th]{Thurston}
W. Thurston, ``Geometry and dynamics of rational functions'',
in ``Complex dynamics: families and friends'', D. Schleicher (Ed.), 2009

\bibitem[T]{T09}
V. Timorin, ``Topological regluing of rational functions'',
Inventiones Math., \textbf{179} (2009), Issue 3, 461--506

\bibitem[W]{Wittner}
B. Wittner, ``On the bifurcation loci of rational maps of degree two'',
PhD Thesis, Cornell University, 1988
\end{thebibliography}
\end{document}